\documentclass[11pt]{article}

\usepackage{amsfonts}
\usepackage{amssymb}
\usepackage{amsmath}
\usepackage{mathrsfs}
\usepackage[colorlinks, linkcolor=blue, anchorcolor=blue, citecolor=blue]{hyperref}
\usepackage{color}
\usepackage{leftidx}
\usepackage{verbatim}
\usepackage{enumerate}
\usepackage{amsthm}
\usepackage{authblk}
\usepackage{geometry}
\usepackage{MnSymbol}
\usepackage[toc,page]{appendix}

\newcommand{\E}{\mathbb{E}}

\newcommand{\R}{\mathbb{R}}
\newcommand{\N}{\mathbb{N}}

\newcommand{\tY}{\widetilde{Y}}

\newcommand{\tZ}{\widetilde{Z}}

\newcommand{\bz}{\mathbf{z}}

\newcommand{\kH}{\mathfrak{H}}

\newcommand{\1}{\mathbf{1}}

\newcommand{\lgl}{\langle}
\newcommand{\rgl}{\rangle}

\newcommand{\cF}{\mathcal{F}}

\newcommand{\tot}{\widetilde{\otimes}}

\newcommand{\wth}{\widetilde{h}}

\newcommand{\hh}{\widehat{h}}
\newcommand{\dP}{\mathbb{P}}

\newcommand{\br}{\mathbf{r}}
\newcommand{\dmd}{\diamond}

\newcommand{\bH}{\mathbf{H}}

\newcommand{\hQ}{\widehat{Q}}
\newcommand{\hB}{\widehat{B}}

\newcommand{\tB}{\widetilde{B}}

\newtheorem{theorem}{Theorem}[section]

\newtheorem{definition}[theorem]{Definition}
\newtheorem{remark}[theorem]{Remark}

\newtheorem{innerhyp}{Hypothesis}
\newenvironment{hyp}[1]
 {\renewcommand\theinnerhyp{#1}\innerhyp}
 {\endinnerhyp}

\begin{document}
\numberwithin{equation}{section}
\title{Intermittency for the Parabolic Anderson model of Skorohod type driven by a rough noise
}

\author{Nicolas Ma \footnote{Email: nma@ku.edu}}
\author{David Nualart  \footnote{Supported by NSF grants DMS 1811181. Email: nualart@ku.edu}}
\author{Panqiu Xia \footnote{Email: pqxia@ku.edu}}

\affil{Department of Mathematics, University of Kansas, Lawrence, KS, 66044, USA.}
\date{}
\maketitle

\begin{abstract}	In this paper, we study the parabolic Anderson model of Skorohod type driven by a fractional Gaussian noise in time with Hurst parameter $H\in (0,1/2)$. By using the Feynman-Kac representation for the $L^p(\Omega)$ moments of the solution, we find the upper and lower bounds for the  moments.
\end{abstract}

\section{Introduction}\label{s.1}
In this paper, we consider the following parabolic Anderson model of Skorohod type
\begin{align}\label{spde}
\frac{\partial}{\partial t}u(t,x)=\frac{1}{2}\Delta u(t,x)+u(t,x)\dmd \frac{\partial}{\partial t}W(t,x),
\end{align}
where $\dmd$ denotes the Wick product. The noise $W=\{W(t,x),(t,x)\in \R_+\times \R^d\}$ is a Gaussian random field, that is a fractional Brownian motion of Hurst parameter $H\in(0,\frac{1}{2})$ in time, and has a correlation in space given by a function $Q$, namely,
\begin{align*}
\E [W(t,x)W(s,y)]=\frac{1}{2}\big(t^{2H}+s^{2H}-|t-s|^{2H}\big)Q(x,y),
\end{align*}
for all $s,t\in \R_+$ and $x,y\in \R^d$. We assume that the covariance function $Q$ satisfies the following conditions:
\begin{hyp}{(H1)}\label{hyp1}
There exist constants $\alpha\in(1-2H,1]$ and $C_1>0$ such that
\begin{align}
Q(x,x)+Q(y,y)-2Q(x,y)\leq C_1|x-y|^{2\alpha},
\end{align}
for all $x,y\in \R^d$.
\end{hyp}
\begin{hyp}{(H2)}\label{hyp2}
There exist constants $\beta\in [0,1)$ and $C_2>0$ such that for any $M>0$,
\begin{align}
\inf_{\displaystyle\min_{i=1,\dots,d}(|x_i|\wedge |y_i|)>M}Q(x,y)\geq C_2M^{2\beta}.
\end{align}
\end{hyp}

A similar equation in the Stratonovich sense, where the Wick product in \eqref{spde} is replaced by the ordinary product, has been studied by Hu et al. \cite{ap-12-hu-lu-nualart} and Chen et al. \cite{ptrf-18-chen-hu-kalbasi-nualart}. In these papers, it has been proved that under Hypotheses \ref{hyp1} and \ref{hyp2}, the Stratonovich type equation with bounded initial condition has a unique solution, which admits a Feynman-Kac representation. Additionally, by using the Feynman-Kac formula for the moments of the solution, the authors in \cite{ptrf-18-chen-hu-kalbasi-nualart} studied the intermittency phenomenon for the solution and obtain the following bounds
\begin{align*}
\underline{C}_x \exp\Big(\underline{C}n^{\frac{2-\beta}{1-\beta}}t^{\frac{2H+\beta}{1-\beta}}\Big)\leq \E \big[u(t,x)^k\big]\leq \overline{C}_x\exp\Big(\overline{C}n^{\frac{2-\alpha}{1-\alpha}}t^{\frac{2H+\alpha}{1-\alpha}}\Big)
\end{align*}
for all $t\geq 1$, $x\in \R^d$ and $n\geq 1$, where  $\underline{C}$ and $\overline{C}$ are positive constants depending on $d,H,\alpha, \|u_0\|_{\infty}$ and $\underline{C}_x,\overline{C}_x>0$ depend on $d,H,\alpha, \|u_0\|_{\infty}$ and $x$.

In this paper, we will study the intermittency for the Skorohod equation \eqref{spde}. The upper bounds for the moments of the solution can be easily obtained. This is due to the fact that the solution to the Skorohod equation is bounded by the solution to the equation of Stratonovich type. For the same reason, to get lower bounds is more involved. By using the Feynman-Kac formula for the moments, we see that in comparison with the Stratonovich case, the exponent in our case contains an additional negative term. This increases the difficulty to estimate lower bounds for the moments. To settle this difficulty, we pin the Brownian motion $B_t$ at the middle point $t/2$, and observe that conditional on $B_{t/2}=r$, $B_s$ is a Brownian bridge before time $t/2$, and an independent Brownian motion after $t/2$. Then, we estimate the probability of the event that the supremum and the H\"{o}lder norm of the Brownian bridge (motion) are bounded above and below by appropriate constants. This allows us to find a lower bound for the moments of the solution.

This paper is organized as follows. In Section \ref{s.2}, we give a brief introduction on the Malliavin calculus and present the precise definition of the solution to equation \eqref{spde}. In Section \ref{s.3}, following the idea of Hu et al. \cite{ap-12-hu-lu-nualart}, we prove that equation \eqref{spde} has a unique solution and give the Feynman-Kac formula and the chaos expansion of the solution. Then, we provide the upper bounds for the moments. Finally, the lower bounds for the moments are proved in Section \ref{s.4}.

\section{Preliminaries}\label{s.2}
Let $W=\{W(t,x),(t,x)\in \R_+\times\R^d\}$ be the Gaussian random field introduced in Section \ref{s.1} defined on a probability space $(\Omega,\cF, \dP)$. Let $\kH$ be the Hilbert space defined as the completion of the linear span of the indicator functions of rectangles of $\R_+\times \R^d$ with respect to the inner product
\[
\lgl \1_{[0,t]\times [0,x]}, \1_{[0,s]\times [0,y]}\rgl_{\kH}=\frac{1}{2}\big(t^{2H}+s^{2H}-|t-s|^{2H}\big)Q(x,y),
\] 
for all $s,t\in \R_+$ and $x=(x_1,\dots,x_d),y=(y_1,\dots,y_d)\in \R^d$, where  $\1_{[0,x]}=\prod_{i=1}^d\1_{[0,x_i]}$ and $\1_{[0,x_i]}=-\1_{[x_i,0]}$ if $x_i< 0$. 
For any function $h\in \kH$, we write
\[
W(h):=\int_0^{\infty}\int_{\R^d}h(t,x) W(dt,dx),
\]
where the integral is the It\^{o}-Wiener integral. In other words, $\{W(h),h\in\kH\}$ is an isonormal Gaussian process on $\kH$, that is, a centered Gaussian family with covariance
\[
\E [W(h) W(\hh)]=\lgl h,\hh\rgl_{\kH},
\]
for all $h, \hh \in \kH$. For any positive integer $n$, we write $H_n$ for the Hermite polynomial on $\R$, that is,
\[
H_n(x)=\frac{(-1)^n}{n!}e^{\frac{x^2}{2}}\frac{d^n}{dx^n}e^{-\frac{x^2}{2}},\quad x\in \R.
\]
Let $\bH_n$ be the closed linear subspace of $L^2(\Omega)$ generated by the set of random variables $\{H_n(W(h)), h\in \kH, \|h\|_{\kH}=1\}$. The space $\bH_n$ is called the $n$-th Wiener chaos. Denote by $\kH^{\otimes n}$ the $n$-fold tensor product space of $\kH$. We write $I_n$ for the isometry map between $\kH^{\otimes n}$ (with the modified norm $\sqrt{n!}\|\cdot\|_{\kH^{\otimes n}}$) and $\bH_n$, given by $I_n(h^{\otimes n})=H_n(W(h))$. It is known (c.f. Lemma 1.1.1 and Theorem 1.1.2 of Nualart \cite{springer-06-nualart}) that
\begin{enumerate}[(i)]
\item $\bH_n$ and $\bH_m$ are orthogonal if $n\neq m$. That is
\[
\E (FG)=0,\quad \forall F\in \bH_n, G\in \bH_m, n\neq m.
\]

\item Any square integrable $W$-measurable random variable $F$ can by uniquely represented as the following orthogonal Wiener chaos expansion
\begin{align}\label{caof}
F=\E (F)+\sum_{n=1}^{\infty}I_n(f_n),
\end{align}
where $f_n\in\kH^{\otimes n}$ are symmetric. 
\end{enumerate}
By above properties and the isometry between $\kH^{\otimes n}$ and $\bH_n$, for any $F\in L^2(\Omega)$ has the chaos expansion \eqref{caof}, the following equality holds
\begin{align*}
\E (F^2)=\E (F)^2+\sum_{n=1}^{\infty}n!\|f_n\|_{\kH^{\otimes n}}^2.
\end{align*}
Let $F,G\in L^2(\Omega)$. Suppose that $F=\E(F)+\sum_{n=1}^{\infty} I_n(f_n)$ and $G=\E(G)+\sum_{m=1}^{\infty} I_m(g_m)$. Then, by definition, the Wick product of $F$ and $G$ can be  written as the following expression, if the last series is convergent in $L^2(\Omega)$,
\[
F\dmd G=E(F)\sum_{m=1}^{\infty}I_m(g_m)+E(G)\sum_{n=1}^{\infty}I_n(f_n)+\sum_{n,m=1}^{\infty}I_{n+m}(f_n\tot g_m),
\]
where $f_n\tot g_m$ is the symmetrization of $f_n\otimes g_m$ in $\kH^{\otimes (n+m)}$.
\begin{remark}
The assumption $F,G\in L^2(\Omega)$ does not imply the convergence of $F\dmd G$. We refer the readers to the book of Hu \cite{ws-16-hu} for a detailed account on the Wick product and  sufficient conditions for the existence of $F\dmd G$.
\end{remark}
Let $u=\{u(t,x),(t,x)\in \R_+\times \R^d\}$ be a $W$-measurable random field. Suppose that $\E [u(t,x)^2]<\infty$ for all $(t,x)\in \R_+\times \R^d$. Then, $u(t,x)$ has a Wiener chaos expansion as follows
\begin{align}\label{preucs}
u(t,x)=\E (u(t,x))+\sum_{n=1}^{\infty}I_n(h_n(\cdot, t,x)).
\end{align}
In the following, we define the Skorohod integral and the solution to the Skorohod type stochastic partial differential equation (SPDE) \eqref{spde}. For more details on this topic, we refer the readers to Hu and Nualart \cite{ptrf-09-hu-nualart}.
\begin{definition}\label{srk}
A square integrable random field $u$ of the the form \eqref{preucs} is called to be Skorohod integrable, if $\E(u)\in \kH$, $h_n\in \kH^{\otimes (n+1)}$ for all $n\geq 1$ and the series
\[
\delta(u)=\int_0^{\infty}\int_{\R^d}u(t,x)\delta W( t, x):=W(\E(u))+\sum_{n=1}^{\infty}I_{n+1}(\wth_n)
\]
converges in $L^2(\Omega)$, where $\wth_n$ is the symmetrization of $h_n$ as an element in $\kH^{\otimes (n+1)}$. The collection of all such random fields is denoted by $Dom(\delta)$.
\end{definition}
\begin{definition}
Let $u_0$ be a bounded measurable function on $\R^d$. A random field $u=\{u(t,x)\in \R_+\times \R^d\}$ is said to be a (mild) solution to the SPDE \eqref{spde} with initial condition $u_0$, if for any $(t,x)\in \R_+\times \R^d$ the random field
\[
\Big\{\1_{[0,t]}(s)\int_{\R^d}p_{t-s}(x-z)u(s,z)\1_{[0,z]}(y)dz,(s,y)\in \R_+\times \R^d\Big\}
\]
is an element of $Dom(\delta)$, and the following equality holds almost surely,
\begin{align*}
u(t,x)=\int_{\R}p_t(x-y)u_0(y)dy+\int_0^{\infty}\int_{\R^d}\Big(\1_{[0,t]}(s)\int_{\R^d}p_{t-s}(x-z)u(s,z)\1_{[0,z]}(y)dz\Big) \delta W( s, y),
\end{align*}
where $p_t(x)=(2\pi t)^{-\frac{d}{2}}e^{-\frac{|x|^2}{2t}}$ denotes the heat kernel on $\R^d$ and the last integral is the Skorohod integral in the sense of Definition \ref{srk}.
\end{definition}

\section{Feynman-Kac formula, chaos expansion and the upper bound}\label{s.3}
Let $B$ be a standard $d$-dimensional Brownian motion independent of $W$. For any $(t,x)\in \R_+\times \R^d$, let $B^x_t=x+B_t$, and let $g^B_{t,x}:\R_+\times \R^d\to \R$ be given by 
\begin{align}\label{gtx}
g_{t,x}^B(r,z):=\1_{[0,t]}(r)\1_{[0,B^x_{t-r}]}(z).
\end{align}
Then due to Theorem 2.2 of Chen et al. \cite{ptrf-18-chen-hu-kalbasi-nualart}, we know that $g_{t,x}\in \kH$.  Since the Feynman-Kac representation for the Stratonovich type equation has been already established in \cite{ptrf-18-chen-hu-kalbasi-nualart}, then by the same argument as in Section 6 of Hu et al. \cite{ap-12-hu-lu-nualart}, we can immediately derive the following theorem.
\begin{theorem}
Suppose that $Q$ satisfies Hypothesis \ref{hyp1}. Let $B$ be a standard $d$-dimensional Brownian motion independent of $W$. For any $(t,x)\in \R_+\times \R^d$, let $g^B_{t,x}$ be defined in \eqref{gtx}. Then for any bounded measurable function $u_0$ on $\R^d$, the process $u=\{u(t,x), (t,x)\in \R_+\times \R^d\}$ given by
\begin{align}\label{fk}
u(t,x)=\E \big[u_0(B_t^x)\exp\big(W(g_{t,x}^B)-\frac{1}{2}\|g_{t,x}^B\|_{\kH}^2\big)\big]
\end{align}
is the unique (mild) solution to \eqref{spde} with initial condition $u_0$.
\end{theorem}
\begin{remark} We can further deduce that $u(t,x)$  has the following chaos expansion,
\begin{align*}
u(t,x)=\sum_{n=0}^{\infty}I_n(h_n(t,x)),
\end{align*}
with
\begin{align*}
h_n(t,x)(\br,\bz)=\frac{1}{n!}\E\big[u_0(B_t^x)g_{t,x}^{B^{1}}(r_1,z_1)\dots g_{t,x}^{B^{n}}(r_n,z_n)\big],
\end{align*}
where $\{B^k\}_{k\geq 1}$ are independent copies of $B$, $\br=(r_1,\dots, r_n)\in \R_+^n$ and $\bz=(z_1,\dots, z_n)\in (\R^d)^n$.
\end{remark}
The next theorem provides an upper bound for moments of the solution to \eqref{spde}.
\begin{theorem}\label{up}
Suppose that $u_0$ is bounded and $Q$ satisfies Hypothesis \ref{hyp1}. Let $u$ be the solution to equation \eqref{spde}. Then for all positive integer $n$, $t\geq 1$ and $x\in\R^d$, the following inequality holds,
\begin{align*}
\E \big[u(t,x)^n\big]\leq C_x\exp \Big(C n^{\frac{2-\alpha}{1-\alpha}}t^{\frac{2H+\alpha}{1-\alpha}}\Big),
\end{align*}
where $C>0$ depends on $d,H,\alpha, \|u_0\|_{\infty}$ and $C_x>0$ depends on  $d,H,\alpha, \|u_0\|_{\infty}$ and $x$.
\end{theorem}
\begin{proof}
Recall that $\{B^k\}_{k\geq 1}$ are independent $d$-dimensional Brownian motions and $g^{B^k}_{t,x}$ is defined in \eqref{gtx}. By the Feynman-Kac formula \eqref{fk}, we can write the moment formula for the solution as follows
\begin{align}\label{moment}
\E [u(t,x)^n]=\E^B\Big[\prod_{k=1}^nu_0(B_t^{k,x})\exp\Big(\frac{1}{2}\sum_{1\leq i\neq j\leq n}\lgl g_{t,x}^{B^i},g_{t,x}^{B^j}\rgl_{\kH}\Big)\Big].
\end{align}
Combining \eqref{moment} and Theorem 3.1 in \cite{ptrf-18-chen-hu-kalbasi-nualart}, we can deduce that
\begin{align*}
\E [u(t,x)^n]\leq \E^B\Big[\prod_{k=1}^nu_0(B_t^{k,x})\exp\Big(\frac{1}{2}\sum_{1\leq i, j\leq n}\lgl g_{t,x}^{B^i},g_{t,x}^{B^j}\rgl_{\kH}\Big)\Big]\leq C_x \exp\Big(Cn^{\frac{2-\alpha}{1-\alpha}}t^{\frac{2H+\alpha}{1-\alpha}}\Big).
\end{align*}
The proof of  this theorem is completed.
\end{proof}
\begin{remark}
An alternative proof of Theorem \ref{up} can be established by the chaos expansion of the solution to the SPDE \eqref{spde} and the hypercontractivity property of fixed Wiener chaos (c.f. Hu et al. \cite{abel-18-hu-huang-le-nualart-tindel}).
\end{remark}

\section{Lower bound for the moments}\label{s.4}
In this section, we prove the following theorem, which provides a lower bound for the moments of the solution to the SPDE \eqref{spde}.
\begin{theorem}\label{lp}
Suppose that $u_0$ is bounded, $\inf_{x\in\R^d}u_0>0$, and $Q$ satisfies Hypotheses \ref{hyp1} and \ref{hyp2} with $\alpha=\beta$. Let $u$ be the solution to equation \eqref{spde}. Then there exists a positive integer $N$ depending on $d,H$ and $\alpha$, such that for all $n\geq N$, $t\geq 1$ and $x\in\R^d$, the following inequality holds,
\begin{align}\label{lb}
\E \big[u(t,x)^n\big]\geq C_x\exp \Big(C n^{\frac{2-\alpha}{1-\alpha}}t^{\frac{2H+\alpha}{1-\alpha}}\Big),
\end{align}
where $C>0$ depends on $d,H,\alpha, \|u_0\|_{\infty}, \inf_{x\in\R^d}u_0$ and $C_x>0$ depends on  $d,H,\alpha, \|u_0\|_{\infty}$, $\inf_{x\in\R^d}u_0$ and $x$.
\end{theorem}

\begin{proof}
We follow the ideas of Chen et al. \cite{ptrf-18-chen-hu-kalbasi-nualart} and Hu et al. \cite{abel-18-hu-huang-le-nualart-tindel} to prove this theorem. Without loss of generality, we assume that $u_0\equiv 1$. Recall that $\{B^k\}_{k\geq 1}$ are independent $d$-dimensional Brownian motions and $g^{B^k}_{t,x}$ is defined in \eqref{gtx}. By the moment formula \eqref{moment} and Lemma 4.2 of \cite{ptrf-18-chen-hu-kalbasi-nualart}, there exist a Gaussian process $X=\{X(x),x\in\R^d\}$ with correlation $\E [X(x)X(y)]=Q(x,y)$ and an independent fractional Brownian motion $\hB=\{\hB_t,t\in \R\}$ with Hurst parameter $H$, such that
\begin{align}\label{moment1}
\E [u(t,x)^n]=\E^B\exp\bigg\{\E^{X,\hB}\Big[\frac{1}{2}\Big(\int_0^t\sum_{i=1}^nX(B_{t-s}^{i,x})d\hB_s\Big)^2\Big]-\frac{1}{2}\sum_{i=1}^n\|g_{t,x}^{B^i}\|_{\kH}^2\bigg\}.
\end{align}
Due to Lemma 4.3 of \cite{ptrf-18-chen-hu-kalbasi-nualart}, we know that there exists a constant $C_H>0$ depending on $H$ such that
\begin{align}
\E^{X,\hB}\Big[\frac{1}{2}\Big(\int_0^t\sum_{i=1}^nX(B_{t-s}^{i,x})d\hB_s\Big)^2\Big]\geq C_H\Big[\int_0^t\Big(\sum_{i,j=1}^nQ(B_s^{i,x},B_s^{j,x})\Big)^{\frac{1}{2H}}ds\Big]^{2H}.
\end{align}
On the other hand, by (2.2) and (2.12) of \cite{ptrf-18-chen-hu-kalbasi-nualart}, we have
\begin{align}\label{ipch}
\|g_{t,x}^{B^{i,x}}\|_{\kH}^2=\E \big[I_1(g_{t,x}^{B^i,x})^2\big]=&H\int_0^t\theta^{2H-1}[Q(B^{i,x}_{\theta},B^{i,x}_{\theta})+Q(B^{i,x}_{t-\theta},B^{i,x}_{t-\theta})]d\theta\nonumber\\
&+|\alpha_H|\int_0^t\int_0^{\theta}r^{2H-2}\hQ(\theta,\theta-r,B^{i,x},B^{i,x})drd\theta,
\end{align}
where $\alpha_H=2H(2H-1)$ and 
\begin{equation*}
\hQ(u,v,\phi,\psi)=\frac{1}{2}[Q(\phi_u,\psi_u)+Q(\phi_v,\psi_v)-Q(\phi_u,\psi_v)-Q(\phi_v,\psi_u)].
\end{equation*}
Recall that $Q$ satisfies Hypothesis \ref{hyp1}. Thus it is easy to deduce that
\begin{align}\label{hq}
\hQ(\theta, \theta-r, B^{i,x}, B^{i,x})\leq  \frac{C_1}{2}|B^i_{\theta}-B^i_{\theta -r}|^{2\alpha}
\end{align}
and
\begin{align}\label{qxy0}
|Q(x,y)|\leq (C_1^{1/2}|x|^{\alpha}+Q(0,0)^{1/2})(C_1^{1/2}|y|^{\alpha}+Q(0,0)^{1/2}).
\end{align}
To simplify the computations, we assume that $Q(0,0)=0$. In the general case, the proof can be done in a similar way without significant differences. Let $M>0$ and let $\epsilon\in(0,\frac{1}{2})$. Consider the following events
\begin{align*}
G_0^1(M)=\bigg\{\inf_{\substack{1\leq i\leq n, 1\leq j\leq d\\  s\in [t/2,t]}}|B_s^{i,x,j}|\geq M\bigg\},\quad G_0^2(M)=\bigg\{ \sup_{\substack{1\leq i\leq n,1\leq j\leq d\\s\in [0,t]}}|B_s^{i,x,j}|\leq 4M\bigg\},\\
\end{align*}
and
\begin{align*}
 G_0^3(M)=\bigg\{\sup_{\substack{1\leq i\leq n,1\leq j\leq d\\ 0\leq v<u\leq t}}\frac{|B^{i,x,j}_u-B^{i,x,j}_v|}{|u-v|^{\frac{1}{2}-\epsilon}}\leq\frac{16M}{t^{\frac{1}{2}-\epsilon}} \bigg\},
\end{align*}
where $B^{i,x,j}$ denotes the $j$-th component of $B^{i,x}$ for $j=1,\dots, d$.

 On $G_0^1(M)$, by Hypothesis \ref{hyp2} and using the assumption that $\alpha=\beta$, we have the inequality
\begin{align}\label{ex1}
\Big[\int_0^t\Big(\sum_{i,j=1}^nQ(B_s^{i,x},B_s^{j,x})\Big)^{\frac{1}{2H}}ds\Big]^{2H}\geq& \Big[\int_{\frac{t}{2}}^t\big(C_2n^2|M|^{2\alpha }\big)^{\frac{1}{2H}}ds\Big]^{2H}\nonumber\\
=&2^{-2H}C_2n^2M^{2\alpha}t^{2H}.
\end{align}
On $G_0^2(M)$, using \eqref{qxy0}, we get that
\begin{align}
\int_0^t\theta^{2H-1}[Q(B^{i,x}_{\theta},B^{i,x}_{\theta})+Q(B^{i,x}_{t-\theta},B^{i,x}_{t-\theta})]d\theta\leq &\int_0^t\theta^{2H-1}2C_1|4\sqrt{d}M|^{2\alpha}d\theta\nonumber\\
=&2^{4\alpha}d^{\alpha}H^{-1}C_1M^{2\alpha}t^{2H}.
\end{align}
Finally, on $G_0^3(M)$, using \eqref{hq}, we get
\begin{align}\label{ex3}
\int_0^t\int_0^{\theta}r^{2H-2}\hQ(\theta,\theta-r,B^{i,x},B^{i,x})drd\theta\leq &\int_0^t\int_0^{\theta}r^{2H-2}\frac{C_1}{2}\Big(\frac{16\sqrt{d}M}{t^{\frac{1}{2}-\epsilon}}r^{\frac{1}{2}-\epsilon}\Big)^{2\alpha}drd\theta\nonumber\\
=&\frac{2^{8\alpha-1}d^{\alpha}C_1M^{2\alpha}t^{2H}}{(2H+\alpha-2\alpha\epsilon-1)(2H+\alpha-2\alpha\epsilon)}.
\end{align}
Set $G_0(M)=\bigcap_{k=1}^3G_0^k(M)$. Due to inequalities \eqref{moment1} - \eqref{ipch} and \eqref{ex1} - \eqref{ex3}, we obtain
\begin{align}\label{unl0}
\E [u(t,x)^n]\geq&\exp\big[(c_1n^2-c_2n)M^{2\alpha}t^{2H}\big]  \dP[G_0(M)],
\end{align}
where 
\[
c_1=2^{-2H}C_2C_H\quad \mathrm{and}\quad c_2=2^{4\alpha-1}d^{\alpha}C_1+\frac{2^{8\alpha-2}d^{\alpha}C_1|\alpha_H|}{(2H+\alpha-2\alpha\epsilon-1)(2H+\alpha-2\alpha\epsilon)}.
\]

For any $x=(x_1,\dots, x_d)\in \R^d$, let $\{\tB^{j,x_j}\}_{1\leq j\leq d}$ be independent one-dimensional Brownian motions such that $\tB^{j,x_j}$ starts from $x_j$ for all $j=1,\dots, d$. For any $j$, let $G^j(M)$ be the event given by
\begin{align}\label{gm}
G^{j}(M):=\Big\{\inf_{s\in[t/2,t]}|\tB^{j,x_j}_s|\geq M, \sup_{s\in [0,t]}|\tB^{j,x_j}_s|\leq 4M, \sup_{0\leq v<u\leq t}\frac{|\tB^{j,x_j}_u-\tB^{j,x_j}_v|}{|u-v|^{\frac{1}{2}-\epsilon}}\leq \frac{16M}{t^{\frac{1}{2}-\epsilon}}\Big\},
\end{align}
and denote $G(M)=\bigcap_{j=1}^d G^j(M)$. Since $\{B^{i,x}\}_{1\leq i\leq n}$ are independent $d$-dimensional Brownian motions starting at $x=(x_1,\dots,x_d)$, the following equality holds
\begin{align*}
\dP[G_0(M)]=\dP [G(M)]^n=\prod_{j=1}^d\dP[G^j(M)]^n.
\end{align*}
This allows us to rewrite \eqref{unl0} in the following way,
\begin{align}\label{unl1}
\E [u(t,x)^n]\geq&\exp\big[(c_1n^2-c_2n)M^{2\alpha}t^{2H}\big] \prod_{j=1}^d \dP[G^j(M)]^n.
\end{align}

In order to estimate $\dP[G^j(M)]$, we pin the Brownian motion $\tB^{j,x_j}$ at $t/2$, and obtain that
\begin{align}\label{gjm1}
\dP[G^j(M)]=&\int_M^{4M}\dP\big[G^j(M)\big|\tB^{j,x_j}_{t/2}=r\big]q_{t/2}(r-x_j)dr\nonumber\\
\geq &\int_{2M}^{3M}\dP\big[G^j(M)\big|\tB^{j,x_j}_{t/2}=r\big]q_{t/2}(r-x_j)dr,
\end{align}
where $q_t(x)=(2\pi t)^{-\frac{1}{2}}\exp[-x^2/(2t)]$ is the one-dimensional heat kernel. Notice that conditioned on $\tB^{j,x_j}_{t/2}=r$, the process $\{\tB^{j,x_j}_s,s\in[0,t/2]\big\}$ is a Brownian bridge, denoted by $Y=\{Y_s, s\in [0,t/2]\}$, such that $Y_0=x_j$ and $Y_{t/2}=r$. In addition,  the process $\{\tB^{j,x_j}_{t/2+s}-r,s\in[0,t/2]\}$, denoted by $Z=\{Z_s, s\in[0,t/2]\}$, is a standard Brownian motion independent of $Y$. Let $A_1,\dots,A_4$ be the events given by
\begin{align*}
A_1=\Big\{\sup_{s\in[0,t/2]}|Z_s|\leq M\Big\},\quad A_2=\Big\{\sup_{s\in[0,t/2]}|Y_s|\leq 4M\Big\},
\end{align*}
\begin{align*}
A_3=\bigg\{ \sup_{0\leq u< v\leq t/2}\frac{|Z_u-Z_v|}{|u-v|^{\frac{1}{2}-\epsilon}}\leq \frac{8M}{t^{\frac{1}{2}-\epsilon}}\bigg\},\quad A_4=\bigg\{ \sup_{0\leq u< v\leq t/2}\frac{|Y_u-Y_v|}{|u-v|^{\frac{1}{2}-\epsilon}}\leq \frac{8M}{t^{\frac{1}{2}-\epsilon}}\bigg\}.
\end{align*}
Observe that for any $0\leq v<t/2<u\leq t$, it is easy to see that
\begin{align*}
\frac{|\tB^{j,x_j}_u-\tB^{j,x_j}_v|}{|u-v|^{\frac{1}{2}-\epsilon}}\leq 2\max \bigg\{\frac{|\tB^{j,x_j}_{t/2}-\tB^{j,x_j}_v|}{|t/2-v|^{\frac{1}{2}-\epsilon}},\frac{|\tB^{j,x_j}_u-\tB^{j,x_j}_{t/2}|}{|u-t/2|^{\frac{1}{2}-\epsilon}}\bigg\}.
\end{align*}
It follows that conditional on $\tB^{j,x_j}_{t/2}=r$,
\begin{align*}
\sup_{0\leq v<u\leq t}\frac{|\tB^{j,x_j}_u-\tB^{j,x_j}_v|}{|u-v|^{\frac{1}{2}-\epsilon}}\leq 2\max \bigg\{\sup_{0\leq v<u\leq t/2}\frac{|Y_u-Y_v|}{|u-v|^{\frac{1}{2}-\epsilon}},\sup_{0\leq v<u\leq t/2}\frac{|Z_u-Z_v|}{|u-v|^{\frac{1}{2}-\epsilon}}\bigg\},
\end{align*}
and thus
\begin{align}\label{sphd}
&\bigg\{\sup_{0\leq v<u\leq t}\frac{|\tB^{j,x_j}_u-\tB^{j,x_j}_v|}{|u-v|^{\frac{1}{2}-\epsilon}}\leq \frac{16M}{t^{\frac{1}{2}-\epsilon}}\bigg\}\supset A_3\cap A_4.
\end{align}
Moreover, if we restrict $r\in [2M, 3M]$ as in \eqref{gjm1}, the following inclusion is true,
\begin{align}\label{b2}
\Big\{\inf_{s\in [t/2,t]} |\tB^{j,x_j}_s|\geq M,\sup_{s\in [0,t]} |\tB^{j,x_j}_s|\leq 4M\Big\}\supset A_1\cap A_2.
\end{align}
 Therefore, by \eqref{gm}, \eqref{sphd} and \eqref{b2}, we have for $r\in [2M,3M]$,
\begin{align*}
\dP\big[G^j(M)\big|B^{j,x_j}_{t/2}=r\big]\geq \dP\Big(\bigcap_{k=1}^4A_k\Big).
\end{align*}
Because $Y$ and $Z$ are independent, we can write
\begin{align}\label{a0}
\dP\Big(\bigcap_{k=1}^4A_k\Big)=&1-\dP\Big(\bigcup_{k=1}^4A_k^c\Big)\geq 1-\dP\big(A_1^c\bigcup A_2^c)-\dP\big(A_3^c\bigcup A_4^c)\nonumber\\
=&\dP(A_1)\dP(A_2)+\dP(A_3)\dP(A_4)-1.
\end{align}
{\bf Estimation of $\dP(A_1)$:} It follows from Doob's martingale inequality that
\begin{align}\label{a1}
\dP(A_1)=1-\dP(A_1^c)\geq 1-M^{-2}\E (|Z_{t/2}|^2)= 1-\frac{t}{2M^2}.
\end{align}
{\bf Estimation of $\dP(A_3)$:} Recall that $\epsilon\in (0,1/2)$. By Kolmogorov's continuity criterion (c.f. Theorem 3.1 of Friz and Hairer \cite{springer-14-friz-hairer}), there exists a modification of $Z$, denoted by $\tZ$, and a random variable $K_{\epsilon}$, such that 
\[
 \sup_{0\leq u< v\leq t/2}\frac{|\tZ_u-\tZ_v|}{|u-v|^{\frac{1}{2}-\epsilon}}\leq K_{\epsilon}\quad \mathrm{and}\quad \E(|K_{\epsilon}|^{\frac{2}{\epsilon}})\leq C_{\epsilon}t^2,
\]
where $C_{\epsilon}>0$ is a constant depending only on $\epsilon$. Combining this fact with Chebyshev's inequality, we have
\begin{align}\label{a3}
\dP(A_3)=1-\dP(A_3^c)\geq 1- \Big(\frac{8M}{t^{\frac{1}{2}-\epsilon}}\Big)^{-\frac{2}{\epsilon}}\E(|K_{\epsilon}|^{\frac{2}{\epsilon}})\geq  1-2^{-\frac{6}{\epsilon}}C_{\epsilon}\frac{ t^{\frac{1}{\epsilon}}}{M^{\frac{2}{\epsilon}}}.
\end{align}
{\bf Estimation of $\dP(A_2)$:} Let $\tB$ be a one-dimensional standard Brownian motion. Then the Brownian bridge $Y$ has the same distribution as the process $\tY=\{\tY_s,0\leq s\leq t/2\}$ where
\begin{align}\label{tys}
\tY_s=x_j+\tB_s-\frac{2s}{t}(\tB_{t/2}-r+x_j).
\end{align}
Thus, we can deduce that
\begin{align*}
\dP(A_2^c)=&\dP\Big[\sup_{0\leq s\leq t/2}\Big|\Big(1-\frac{2s}{t}\Big)x_j+\frac{2sr}{t}+\tB_s-\frac{2s}{t}\tB_{t/2}\Big|>4M\Big]\\
\leq &\dP\Big(\sup_{0\leq s\leq t/2}|\tB_s|+|\tB_{t/2}|>4M-r-|x_j|\Big)\leq \dP\Big(\sup_{0\leq s\leq t/2}|\tB_s|>2M-\frac{r+|x_j|}{2}\Big).
\end{align*}
Assume that $\frac{M}{2}>\max\{|x_1|,\dots, |x_n|\}$ and recall that $r\in [2M, 3M]$. It follows that
\begin{align}\label{a2}
\dP(A_2)=1-\dP(A_2^c)\geq 1- \dP\Big(\sup_{0\leq s\leq t/2}|\tB_s|>\frac{M}{4}\Big)\geq 1-\frac{8t}{M^2}.
\end{align}
{\bf Estimation of $\dP(A_4)$:} Due to \eqref{tys} and the fact $r\in [2M, 3M]$, we have
\begin{align*}
\frac{|\tY_u-\tY_v|}{|u-v|^{\frac{1}{2}-\epsilon}}\leq &\frac{|\tB_u-\tB_v|}{|u-v|^{\frac{1}{2}-\epsilon}}+\frac{2|u-v|^{\frac{1}{2}+\epsilon}}{t}(|\tB_{t/2}|+r+|x_j|)\\
\leq &K^{\tB}_{\epsilon}+\frac{2^{\frac{1}{2}-\epsilon}(|\tB_{t/2}|+\frac{7}{2}M)}{t^{\frac{1}{2}-\epsilon}},
\end{align*}
for all $0\leq v<u\leq t/2$, where $K^{\tB}_{\epsilon}$ is the almost surely upper bound of the $(\frac{1}{2}-\epsilon)$-H\"{o}lder norm of $\tB$ on $[0,t/2]$ and $\E [|K^{\tB}_{\epsilon}|^{\frac{2}{\epsilon}}]\leq C_{\epsilon}t^2$. Therefore, 
\begin{align}\label{a4}
\dP(A_4)=&1-\dP(A_4^c)= 1-\dP\Big(\sup_{0\leq u< v\leq t/2}\frac{|\tY_u-\tY_v|}{|u-v|^{\frac{1}{2}-\epsilon}}\geq \frac{8M}{t^{\frac{1}{2}-\epsilon}}\Big)\nonumber\\
\geq& 1- \dP \Big(K_{\epsilon}^{\tB}+\frac{2^{\frac{1}{2}-\epsilon}|\tB_{t/2}|}{t^{\frac{1}{2}-\epsilon}}\geq \frac{M}{t^{\frac{1}{2}-\epsilon}}\Big)\geq1- \big[2^{\frac{2}{\epsilon}-1}C_{\epsilon}+2^{\frac{2}{\epsilon}-3}\E|\tB_1|^{\frac{2}{\epsilon}}\big]\frac{t^{\frac{1}{\epsilon}}}{M^{\frac{2}{\epsilon}}}.
\end{align}
According to inequalities \eqref{a1}, \eqref{a3}, \eqref{a2} and \eqref{a4}, and choosing 
\begin{align}\label{m1}
M\geq C_{1,\epsilon}t^{\frac{1}{2}}:=\max\Big\{\Big(\frac{8}{1-\sqrt{3}/2}\Big)^{1/2},\Big(\frac{2^{\frac{2}{\epsilon}-1}C_{\epsilon}+2^{\frac{2}{\epsilon}-3}\E|\tB_1|^{\frac{2}{\epsilon}}}{1-\sqrt{3}/2}\Big)^{\epsilon/2}\Big\}t^{\frac{1}{2}},
\end{align}
we can make $\dP(A_k)\geq \sqrt{3}/2$ for all $k=1,\dots,4$. Thus by \eqref{a0}, we have 
\begin{align}\label{pgm1}
\dP\big[G^j(M)|\tB^{j,x_j}_{t/2}=r\big]=\dP(\bigcap_{k=1}^4A_k)\geq \frac{1}{2}.
\end{align}
Plugging \eqref{pgm1} into inequality \eqref{gjm1} and recalling that $M$ satisfies \eqref{m1} and $\frac{M}{2}\geq |x_j|$, we can write
\begin{align}\label{gim2}
\dP(G^j(M))\geq &\frac{1}{2}\int_{2M}^{3M}dr q_{t/2}(x^j-r)\geq \frac{M}{\sqrt{4\pi t}}e^{-\frac{16M^2}{t}}\geq \frac{C_{1,\epsilon}}{\sqrt{4\pi}}e^{-\frac{16M^2}{t}}\geq e^{-\frac{16M^2}{t}}.
\end{align}
Combining \eqref{unl1} and \eqref{gim2}, we have
\begin{align}\label{unl2}
\E [u(t,x)^n]\geq&\exp\big[(c_1n^2-c_2n)M^{2\alpha}t^{2H}-c_3nM^2t^{-1}\big],
\end{align}
where $c_3=16d$. Let $N$ be the smallest integer such that $c_1n-c_2>0$. Then, for any $n\geq N$, by maximizing the function
\[
f(M)=(c_1n^2-c_2n)M^{2\alpha}t^{2H}-c_3nM^2t^{-1},
\]
we find 
\begin{align}\label{m0}
M_0=\big(\alpha(c_1n-c_2)c_3^{-1}t^{2H+1}\big)^{\frac{1}{2-2\alpha}},
\end{align}
such that
\begin{align}\label{fm0}
\sup_{M\geq 0}f(M)=f(M_0)=&(1-\alpha)\alpha^{\frac{\alpha}{1-\alpha}}c_3^{-\frac{\alpha}{1-\alpha}}n(c_1n-c_2)^{\frac{1}{1-\alpha}}t^{\frac{2H+\alpha}{1-\alpha}}\nonumber\\
\geq &(1-\alpha)\alpha^{\frac{\alpha}{1-\alpha}}c_3^{-\frac{\alpha}{1-\alpha}}(c_1-c_2/N)n^{\frac{2-\alpha}{1-\alpha}}t^{\frac{2H+\alpha}{1-\alpha}}.
\end{align}
Notice that for any $t\geq 1$ and $n\geq N$, the number $M_0$ given by \eqref{m0} satisfies the following inequality
\begin{align}\label{m01}
M_0\geq \max\{\big(\alpha(c_1N-c_2)c_3^{-1}t^{2H+\alpha}\big)^{\frac{1}{2-2\alpha}}t^{\frac{1}{2}},\big(\alpha(c_1n-c_2)c_3^{-1}\big)^{\frac{1}{2-2\alpha}}t^{\frac{1}{2}}\}.
\end{align}
Let 
\[
n_0(x):=\max\Big\{N, \frac{c_2\alpha+c_3[2\max\{|x_1|,\dots, |x_d|\}]^{2-\alpha}}{c_1 \alpha},\frac{c_2\alpha+c_3C_{1,\epsilon}^{2-2\alpha}}{c_1\alpha}\Big\}
\]
and let 
\[
t_0(x):=\max\Big\{1,\Big(\frac{c_3[2\max\{|x_1|,\dots, |x_d|\}]^{2-2\alpha}}{\alpha(c_1N-c_2)}\Big)^{\frac{1}{2H+\alpha}}, \Big(\frac{c_3C_{1,\epsilon}^{2-2\alpha}}{\alpha(c_1N-c_2)}\Big)^{\frac{1}{2H+\alpha}}\Big\}.
\]
Then, for any 
\begin{align}\label{l1}
(t,n)\in L_1:= \{(s,m)\in \R_+\times \N, s\geq 1, m\geq n_0(x)\},
\end{align}
or
\begin{align}\label{l2}
(t,n)\in L_2:= \{(s,m)\in \R_+\times \N, s\geq t_0(x), m\geq N\},
\end{align}
by using \eqref{m01}, we have $\frac{M_0}{2}\geq \max\{|x_1|,\dots, |x_d|\}$ and $M_0\geq C_{1,\epsilon} t^{\frac{1}{2}}$. This implies that if $(t,n)\in L_1\cup L_2$, inequality \eqref{unl2} is true when $M$ is replaced by $M_0$. In this case, it follows from \eqref{fm0} that
\begin{align}\label{un1}
\E [u(t,x)^n]\geq e^{f(M_0)} \geq \exp\big[(1-\alpha)\alpha^{\frac{\alpha}{1-\alpha}}c_3^{-\frac{\alpha}{1-\alpha}}(c_1-c_2/N)n^{\frac{2-\alpha}{1-\alpha}}t^{\frac{2H+\alpha}{1-\alpha}}\big].
\end{align}
On the other hand, let $M_1=\max\{2|x_1|,\dots, 2|x_d|, C_{1,\epsilon}t_0(x)^{\frac{1}{2}}\}$. Then for any 
\begin{align}\label{l3}
(t,n)\in L_3:=\{(s,m)\in \R_+\times \N, 1\leq s\leq t_0(x),N\leq m\leq n_0(x)\},
\end{align}
inequality \eqref{unl2} is true when $M$ is replaced by $M_1$. In this case, we can deduce that
\begin{align}\label{un2}
\E [u(t,x)^n]\geq&\exp\big[(c_1n^2-c_2n)M_1^{2\alpha}t^{2H}-c_3nM_1^2t^{-1}]\nonumber\\
\geq &\inf_{\substack{1\leq t\leq t_0(x)\\N\leq n\leq n_0(x)}}\bigg\{\exp\big[(c_1n^2-c_2n)M_1^{2\alpha}t^{2H}-c_3nM_1^2t^{-1}-C_0n^{\frac{2-\alpha}{1-\alpha}}t^{\frac{2H+\alpha}{1-\alpha}}]\bigg\}\nonumber\\
&\times \exp\Big(C_0n^{\frac{2-\alpha}{1-\alpha}}t^{\frac{2H+\alpha}{1-\alpha}}\Big)\nonumber\\
:=&C_x \exp\Big(C_0n^{\frac{2-\alpha}{1-\alpha}}t^{\frac{2H+\alpha}{1-\alpha}}\Big).
\end{align}
Notice that $\{t\geq 1, n\geq N\}=L_1\cup L_2\cup L_3$ where $L_1$, $L_2$ and $L_3$ are defined in \eqref{l1}, \eqref{l2} and \eqref{l3} respectively. Therefore, by \eqref{un1} and \eqref{un2}, we have inequality \eqref{lb}. This completes the proof of this theorem.
\end{proof}

\end{document}